\documentclass{amsproc}

 \newtheorem{teor}{Theorem}[section]
 \newtheorem{cor}[teor]{Corollary}
 \newtheorem{lemma}[teor]{Lemma}
 
 \newtheorem{prop}[teor]{Proposition}
 \theoremstyle{definition}
 
 \theoremstyle{remark}
 \newtheorem{remark}[teor]{Remark}

\numberwithin{equation}{section}

\usepackage{subfigure}

\usepackage{color}
\usepackage{graphics}
\usepackage{graphicx}
\usepackage{epsfig}
\usepackage{amssymb}
\usepackage{verbatim}
\usepackage{graphics,amsmath,amsfonts,amssymb} 
\usepackage{epsfig,verbatim,multicol}

\begin{document}

\title[Galoisian \& Qualitative Approaches to Linear PZ Vector Fields ]{Galoisian and Qualitative Approaches to Linear Polyanin-Zaitsev Vector Fields}

\author[P. Acosta-Hum\'anez]{Primitivo B. Acosta-Hum\'anez}
\address[P. Acosta-Hum\'anez]{Universidad del Sim\'on Bolivar, Barranquilla - Colombia}
\email{primitivo.acosta@unisimonbolivar.edu.co}

\author[A. Reyes-Linero]{Alberto Reyes-Linero}
\address[A. Reyes-Linero]{Universidad del Atl\'antico, Barranquilla - Colombia}
\email{areyeslinero@mail.uniatlantico.edu.co}

\author[J. Rodriguez-Contreras]{Jorge Rodr\'i­guez-Contreras}
\address[J. Rodriguez-Contreras]{Universidad del Norte \& Universidad del Atl\'antico, Barranquilla - Colombia}
\email{jrodri@uninorte.edu.co}

\maketitle

\begin{abstract}

The analysis of dynamical systems has been a topic of great interest for researches mathematical sciences for a long times. The implementation of several devices and tools have been useful in the finding of solutions as well to describe common behaviors of parametric families of these systems. In this paper we study deeply a particular parametric family of differential equations, the so-called \emph{Linear Polyanin-Zaitsev Vector Field}, which has been introduced in a general case in \cite{ARR} as a correction of a family presented in \cite{PY}. Linear Polyanin-Zaitsev Vector Field is transformed into a Li\'enard equation and in particular we obtain the Van Der Pol equation. We present some algebraic and qualitative results to illustrate some interactions between algebra and the qualitative theory of differential equations in this parametric family.


\end{abstract}

\section*{Introduction}
The analysis of dynamical systems has been a topic of great interest for a plenty of mathematician and theoretical physicist since the seminal works of H. Poincar\'e. Every systems is dynamical whether it is changing in the time. Was H. Poincar\'e who introduced the qualitative approach to study dynamical systems, while E. Picard and E. Vessiot introduced an algebraic approach to study linear differential equations based on the Galois theory for polynomials, see \cite{Ap}\\

An important family of dynamic systems, are Van der Pol type systems. The forced Van der Pol chaotic oscillator, which was discovered by Van der Pol and Van der Mark (\cite{VDP}, 1927). The Van der Pol oscillator has a long history of being used in both the physical and biological sciences. For instance, in biology, Fitzhugh \cite{119} and Nagumo \cite{120} extended the Van der Pol equation in a planar field as a model for action potentials of neurons. A detailed study on forced Van der Pol equation is found in \cite{121}.\\

The Handbook of Nonlinear Partial Differential Equations \cite{PY}, a unique reference for scientists and engineers, contains over 3,000 nonlinear partial differential equations with
solutions, as well as exact, symbolic, and numerical methods for solving nonlinear equations. First-, second-, third-, fourth-, and higher-order nonlinear equations and systems of equations are considered. A writing errata presented in one of these problems was corrected in \cite{ARR}. This problem in correct form was used and studied on \cite{Almp}. The differential equation system asociated of this problem called "Polyanin-Zaitsev vector field" \cite{ARR}, has as associated foliation a Lienard equation, that is to say, it is closely related to a problem of type Van Der Pol.\\

In this paper we study from algebraic and qualitative point of view one parametric family of linear differential systems. Such parametric family comes from the correction of Exercise 11 in \cite[\S 1.3.3]{PY}. Which we called  Polyanin-Zaitsev vector field, see \cite{ARR}. We find the critical points describing to behavior of them near to. In the algebraic aspects, we obtain the explicit first integral through Darboux method. Moreover, we compute the differential Galois group associated to such systems.

\section{Preliminaries}

A polynomial system in the plane of degree $n$ is given by
	
	\begin{equation}\label{camlien1}
		\begin{array}{lll}
			\dot{x}&=&P(x,y)\bigskip\\
			\dot{y}&=&Q(x,y),
		\end{array}
	\end{equation}
	
	where $P,Q\in\mathbb{C}[x,y]$ (set of polynomials in two variables) and $n$ is the absolute degree of the polynomials $P$ and $Q$.\\

	The polynomial vector field associated with the system \eqref{camlien1} is given by $ X:=(P,Q)$, which can also be written as:
	$$\mathrm{X}=P(x,y)\frac{\partial }{\partial x}+Q(x,y)\frac{\partial }{\partial y}. \hspace{1cm} \label{camvet1}$$\medskip

	A foliation of a polynomial vector field of the form \eqref{camlien1} is given by $$\frac{dy}{dx}=\frac{Q(x,y)}{P(x,y)}$$.

Given the family of equations  $$yy'=(a(2m+k)x^{m+k-1}+b(2m-k)x^{m-k-1})y-(a^2mx^{4k}+cx^{2k}+b^2m)x^{2m-2k-1}$$\bigskip
Then the system of equations associated:
\begin{displaymath}
        \begin{array}{ccl}
		\dot{x}&=&y \bigskip \\
        \dot{y}&=& (a(2m+k)x^{m+k-1}+b(2m-k)x^{m-k-1})y-(a^2mx^{4k}+cx^{2k}+b^2m)x^{2m-2k-1}
    	\end{array}
    \end{displaymath}
see \cite{ARR}.\bigskip\\

We can see two important theorems for the study of infinity behavior, the theorem 1 (see \cite[\S 3.10-P 271]{PK}), which allows us to find the infinity critical points and the theorem 2 (see \cite[\S 3.10-P 272, 273]{PK}), which allows us tu characterize this points.\\     

The following theorem is also of vital importance for our study and can be seen in greater detail in \cite{Cll} and \cite{CPP}.\\

\begin{teor}[Darboux]
	Suppose that the polynomial system \eqref{camlien1} of degree $m$ admits $p$ irreducible algebraic invariant curves
$f_i=0$ with cofactors $K_i$ for $i=1,...,p$; $q$ exponential factor $F_j=exp(g_j/h_j)$ with cofactors $L_j$ for $j=1,...,q$ and $r$ independents single points $(x_k,y_k)\in \mathbb{C}^2$ such that $f_i(x_k,y_k)\neq 0$  for $i=1,...,p$ and $k=1,...,p$,
	also all $h_j$ Factor in factor product $f_1,f_2,...,f_q$ except if it is equal to $1$. So the following statements are kept:\\
	
	a) Exists $\lambda_i, \mu_i\in \mathbb{C}$ not all zero such that
	$$\sum_{i=1}^{p}\lambda_iK_i+\sum_{j=1}^{q}\mu_jL_j=0, \hspace{1cm} (D_{fi})$$
	if only if the function (multi-valued)
	$$H(x,y)=f_1^{\lambda_1}...f_p^{\lambda_p}F_1^{\mu_1}...F_q^{\mu_q}, \hspace{1cm} (2)$$
	is a first integral of the system\eqref{camlien1}. Moreover for real systems the function (2) is real.\\
	
	b)If $p+q+r=[m(m+1)/2]+1$ then exists $\lambda_i, \mu_i\in \mathbb{C}$ not all zero such that they satisfied the condition $(D_{fi})$.\\
	
	c)If $p+q+r=[m(m+1)/2]+2$ then the system \eqref{camlien1} have a rational first integral and all orbits of the system is in some invariant algebraic curve.\\
	
	d) Exists $\lambda_i, \mu_i\in \mathbb{C}$ not all zero such that 
	$$\sum_{i=1}^{p}\lambda_iK_i+\sum_{j=1}^{q}\mu_jL_j+div(P,Q)=0, \hspace{1cm} (D_{if})$$
	if only if the function (2) is a first integral of the system \eqref{camlien1}. Moreover for the real systems the function (2) is real.\\
	
	e)If $p+q+r=m(m+1)/2$ and the $r$ singular points independents are son weak the exist $\lambda_i, \mu_i\in \mathbb{C}$ not all zero such that they satisfied any of the conditions  $(D_{fi})$ o $(D_{if})$.\\
	
	f) Exist  $\lambda_i, \mu_i\in \mathbb{C}$ not all zero such that 
	$$\sum_{i=1}^{p}\lambda_iK_i+\sum_{j=1}^{q}\mu_jL_j+s=0, \hspace{1cm} (D_{in})$$
	with $s \in \mathbb{C}-\{0\}$ if only if the function (multi-valued)
	$$I(x,y,t)=f_1^{\lambda_1}...f_p^{\lambda_p}F_1^{\mu_1}...F_q^{\mu_q}exp(st), \hspace{1cm} (3)$$
	is an invariant of the system \eqref{camlien1}. Moreover for the real systems function (3) is real.\\
\end{teor}

Now we write the Polyanin-Zaitsev vector field, introduced in \cite{ARR}:

 \begin{equation}\label{PZVF}
 	X:=y\frac{\partial}{\partial x}+\left((\alpha x^{m+k-1}+\beta x^{m-k-1})y-\gamma x^{2m-2k-1}\right) \frac{\partial}{\partial x},
 \end{equation}
 
where $\alpha=a(2m+k)$, $\beta=b(2m-k)$ y $\gamma (x)=a²mx^{4k}+cx^{2k}+b²m$.\\

The differential system associated to the Polyanin-Zaitsev vector field \eqref{PZVF} is:  

\begin{equation}\label{din1}
\begin{array}{ccl}
\dot{x}&=&y \bigskip \\
\dot{y}&=& (\alpha x^{m+k-1}+\beta x^{m-k-1})y-\gamma x^{2m-2k-1}.
\end{array}
\end{equation}

The aim of this paper is to analyze from galoisian and qualitative point of view the complete set of families in where the Polyanin-Zaitsev differential system \eqref{din1} is a linear differential system.

\section{Conditions for the problem.}

The following lemma allows us to identify the linear cases associated to Polyanin-Zaitsev vector field.

\begin{lemma}
	The Polyanin-Zaitsev differential system is a linear system, with $Q$ not null polynomial, if is equivalently affine to one of the following families: 

		\begin{equation}
		\begin{array}{ll}\label{f1}
		\dot{x}&=y \\
		\dot{y}&=-cx
		\end{array}
		\end{equation}
		With critical point $(0,0)$, saddle node.
		
				\begin{equation}\label{f2}
		\begin{array}{ll}
		\dot{x}&=y  \\
	\dot{y}&=b(k+2)y-b^2(k+1)x.
		\end{array}
		\end{equation}
		With critical point $(0,0)$.
		
				\begin{equation}\label{f3}
		\begin{array}{ll}
		\dot{x}&=y \\
		\dot{y}&=a(2-k)y-(1-k)a^2x
		\end{array}
		\end{equation}
		With critical point $(0,0)$.
		
				\begin{equation}\label{f4}
		\begin{array}{ll}
		\dot{x}&=y \\
		\dot{y}&=2by-b^2x-cx.
		\end{array}
		\end{equation}
		With critical point $(0,0)$.
		 
		\begin{equation}\label{f5}
		\begin{array}{ll}
		\dot{x}&=y \\
		\dot{y}&=\frac{3}{2}by-\frac{b^2}{2}x-c.
		\end{array}
		\end{equation}
			
		 With critical point of the form $(\frac{-2c}{b^2},0)$.

		\begin{equation}\label{f6}
		\begin{array}{ll}
		\dot{x}&=y \\
		\dot{y}&=2ay-a^2x-cx.
		\end{array}
		\end{equation}
		
		The critical point is $(0,0)$.
		
			\begin{equation}\label{f7}
		\begin{array}{ll}
		\dot{x}&=y \\
		\dot{y}&=\frac{3}{2}ay-\frac{a^2}{2}x-c.
		\end{array}
		\end{equation}
		The critical point are of the form $(\frac{-2c}{a^2},0)$.
		
		\begin{equation}\label{f8}
		\begin{array}{ll}
		\dot{x}&=y \bigskip \\
		\dot{y}&=2(a+b)y-(a^2+b^2)x
		\end{array}
		\end{equation}
			With critical point $(0,0)$.
		
				\begin{equation}\label{sismin}
		\begin{array}{ll}
		\dot{x}&=y \bigskip \\
		\dot{y}&=2(a+b)y-(a^2+c+b^2)x
		\end{array}
		\end{equation}
			With critical point $(0,0)$.
\end{lemma}

\begin{proof}
We will to analyze every possible cases for the constants $a,b$ y $c$.   	

	\begin{itemize}
	\item[Case 1:] If $a=b=0$ and $c\neq 0$, the system \eqref{din1} it is reduce to:
	\begin{align*}
	\dot{x}&=y \\
	\dot{y}&=-cx^{2m-1}.
	\end{align*}	
	Guiven the conditions of the problem $2m-1=1$, the tha associated system is
	\begin{align*}
	\dot{x}&=y \\
	\dot{y}&=-cx.
	\end{align*}

	\item[Case 2:] If $a=c=0$ and $b\neq 0$.
	In this case the system \eqref{din1} it is reduce to:
	\begin{align*}
	\dot{x}&=y \\
	\dot{y}&=b(2m-k)yx^{m-k-1}-b^2mx^{2m-2k-1}.
	\end{align*}
	
	It system is linear if $m-k-1=0$ it's to sat that $2m-2k-1=1$, $2m-k=k+2$ and $m=k+1$, then the associated system is:
	\begin{align*}
	\dot{x}&=y \\
	\dot{y}&=b(k+2)y-b^2(k+1)x.
	\end{align*}	  	
	
	\item[Case 3:] If $b=c=0$ and $a\neq 0$.
		In this case the system \eqref{din1} it is reduce to:
	\begin{align*}
	\dot{x}&=y \\
	\dot{y}&=a(2m+k)yx^{m+k-1}-a^2mx^{2m+2k-1}.
	\end{align*}
	
		It system is linear if $m+k-1=0$ it's to sat that $2m+2k-1=1$, $2m+k=2-k$ and $m=1-k$, then the associated system is
	\begin{align*}
	\dot{x}&=y \\
	\dot{y}&=a(2-k)y-a^2(1-k)x.
	\end{align*}
	
	\item[Case 4:] If $a=0$, $b\neq 0$ and $c\neq 0$, the system \eqref{din1} it is reduce to:
	\begin{align*}
	\dot{x}&=y \\
	\dot{y}&=b(2m-k)yx^{m-k-1}-b^2mx^{2m-2k-1}-cx^{2m-1}.
	\end{align*}
	
	$m-k-1=0$ then $2m-2k-1=1$, but in this case we have two options $2m-1=0$ or $2m-1=0$, it's to say that, the two linear system are:
	
	\begin{align*}
	\dot{x}&=y \\
	\dot{y}&=2by-b^2x-cx.
	\end{align*}

	and 	
	\begin{align*}
	\dot{x}&=y \\
	\dot{y}&=\frac{3}{2}by-\frac{b^2}{2}x-c.
	\end{align*}

	\item[Case 5:] If $b=0$, $a\neq 0$ and $c\neq 0$, the system \eqref{din1} it is reduce to:
	\begin{align*}
	\dot{x}&=y \\
	\dot{y}&=a(2m+k)yx^{m+k-1}-a^2mx^{2m+2k-1}-cx^{2m-1}.
	\end{align*}
	
	$m+k-1=0$ then $2m+2k-1=1$, and we have two options $2m-1=0$ or $2m-1=0$, it's to say that, the two linear system are:
	
	\begin{align*}
	\dot{x}&=y \\
	\dot{y}&=2ay-a^2x-cx.
	\end{align*}
	
	and 	
	\begin{align*}
	\dot{x}&=y \\
	\dot{y}&=\frac{3}{2}ay-\frac{a^2}{2}x-c.
	\end{align*}

	\item[Case 6:] If $c=0$, $a\neq 0$ and $b\neq 0$, the system \eqref{din1} it is reduce to:
	\begin{align*}
	\dot{x}&=y \\
	\dot{y}&=a(2m+k)yx^{m+k-1}-a^2mx^{2m+2k-1}+b(2m-k)yx^{m-k-1}-b^2mx^{2m-2k-1}.
	\end{align*}
	
	In this case $m+k-1=0$ and $m-k-1=0$, then $m=1$ and $k=0$. The asociated system is:
	
	\begin{align*}
		\dot{x}&=y\\
		\dot{y}&=2(a+b)y-(a^2+b^2)x
	\end{align*}

	\item[Case 7:] $c\neq 0$, $a\neq 0$ and $b\neq 0$, then we have the system \eqref{din1}. Again we have that $m+k-1=0$ and $m-k-1=0$, then resultant system is 
	\begin{align*}
	\dot{x}&=y\\
	\dot{y}&=2(a+b)y-(a^2+c+b^2)x
	\end{align*}
	 
	\end{itemize}
\end{proof}

A qualitative and Galoisian detailed study is carried out over the linear system \eqref{sismin} of the previous lemma, the rest of cases are studied in a similar form.\\

\section{Critical Points}

\begin{prop}
	
The following statements hold:
\begin{itemize}
	\item[i.]{If $c>0$, $(0,0)$ it is the only critical point on the finite plane, it is also stable.}
	\item[ii.]{If $a>0$, $b>0$ and $c>0$, the critical points at infinity for the system \eqref{sismin} are given by:
		$$(x,y,0)=(\pm\sqrt{1-y^2},(a+b)\pm\sqrt{2ab-c},0).$$}
\end{itemize}
\end{prop}

\begin{proof}
	Let us analyze now for \eqref{sismin} the critical points in both the finite and infinity planes. For this the the theorem 1 (see \cite[\S 3.10-P 271]{PK}) is used, which is based on the study of the system, on the equator of the Poincar\'e sphere, that is when $x^2+y^2=1$.\\
	
	Remember that $P(x,y)=y$ y $Q(x,y)=2(a+b)y-(a^2+b^2+c)x$ with $a,b,c>0$, then the critical points are given by the system:\\
	
	$$\begin{array}{r}
	y=0\bigskip\\
	2(a+b)y-(a^2+b^2+c)x=0
	\end{array}$$
	
	then the only critical point in the finite plane is $(x,y)=(0,0)$. Now for the qualitative analysis of the critical point we find the Jacobian matrix evaluated in the point, given by:

	$$D(P,Q)\mid_{(0,0)}=\begin{pmatrix}
	0 & 1\bigskip\\
	-(a^2+b^2+c) & 2(a+b)
	\end{pmatrix}$$
	
	The characteristic polynomial is $$\lambda^2-2(a+b)\lambda+(a^2+b^2+c)=0$$ 
	
	then the eigenvalues are:
	\begin{equation}\label{vpo}
	\lambda_{1,2}=\frac{2(a+b)\pm\sqrt{4(a+b)^2-4(a^2+b^2+c)}}{2}=(a+b)\pm\sqrt{2ab-c}.\bigskip\\
	\end{equation}

	According to the theorem 1 (see \cite[\S 3.10-P 271]{PK}) and the theorem 2 (see \cite[\S 3.10-P 272, 273]{PK}), to find the critical points and their behavior we can use the system:
	
	$$\begin{array}{l}
	\pm \dot{y}=yz^rP(1/z,y/z)-z^rQ(1/z,y/z)\bigskip\\
	\pm \dot{z}=z^{r+1}P(1/z,y/z)
	\end{array}$$
	
	in our case $r=1$, when replacing we find:
	\begin{equation}\label{sisinf}\begin{array}{l}
	\pm \dot{y}=y^2-2(a+b)y+(a^2+b^2+c)\bigskip\\
	\pm \dot{z}=zy
	\end{array}\end{equation}
	
	The critical points for this system result from solving the system:
	$$\begin{array}{l}
	y^2-2(a+b)y+(a^2+b^2+c)=0\bigskip\\
	z=0
	\end{array}$$
	
	then
	\begin{equation}\label{evp2}
	y_{1,2}=\frac{2(a+b)\pm\sqrt{4(a+b)^2-4(a^2+b^2+c)}}{2}=(a+b)\pm\sqrt{2ab-c}.
	\end{equation}
	
	Also, as we are in the equator of the Poincar\'e sphere$x^2+y^2=1$ that is $x=\pm\sqrt{1-y^2}$.
	
	Note also that by changing the signs of $ x $ e $ and $ we obtain points that are called \emph{Antipodes}, which have a contrary stability. Taking into account the above conditions, the critical points at infinity are:
	$$(x,y,0)=(\pm\sqrt{1-y^2},(a+b)\pm\sqrt{2ab-c},0).$$
	
	The matriz Jacobian \eqref{sisinf} is:
	$$D(P,Q)=\begin{pmatrix}
	2y-2(a+b)& 0\bigskip\\
	z & 0
	\end{pmatrix}$$
	
	in our case:
	
	$$D(P,Q)\mid_{(y_{1,2},0)}=\begin{pmatrix}
	\pm2\sqrt{2ab-c}& 0\bigskip\\
	0 & (a+b)\pm\sqrt{2ab-c}
	\end{pmatrix}$$
	
	We can notice then that the signs of the own values depend on the signs of the elements in the diagonal. The foregoing indicates that several cases must be analyzed for both the finite and infinity planes, be $\Delta=2ab-c$.
	
	\begin{enumerate}
		\item[Case 1.]{If $\Delta>0$ then in \eqref{vpo} $\lambda_1>0$, but for $\lambda_2$ we have three options. Pero si $\lambda_2<0$ and $\lambda_2=0$, contradicts the fact that $a,b,c>0$ then the only option is $\lambda_2>0$. Since both eigenvalues are positive then the origin is a critical point \emph{repulsor}.\\
			
			Regarding the critical points at infinity in this case you can see that the elements on the diagonal of \eqref{sisinf} have two options $y_1,y_2$. For $y_1$:
			
			$$D(P,Q)\mid_{(y_1,0)}=\begin{pmatrix}
			2\sqrt{2ab-c}& 0\bigskip\\
			0 & (a+b)+\sqrt{2ab-c}
			\end{pmatrix}$$
			
			or this point the values in the diagonal are positive then in infinity this critical point is type \emph{repulsor} and therefore its antipode is type \emph{attractor}.
			
			Para $y_2$:
			$$D(P,Q)\mid_{(y_2,0)}=\begin{pmatrix}
			-2\sqrt{2ab-c}& 0\\
			0 & (a+b)-\sqrt{2ab-c}
			\end{pmatrix}$$
			
			where $-2\sqrt{2ab-c}<0$ and $(a+b)-\sqrt{2ab-c}=\lambda_2>0$ therefore the critical point is type \emph{sadle}. So the phase portrait for the system in this case is:\\
			
			\begin{center}
				\includegraphics[width=4cm]{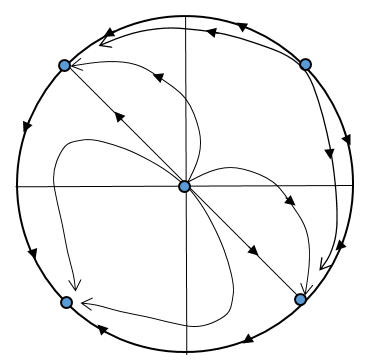} \\Global phase portrait for case 1 \eqref{sismin}
			\end{center}
			
		}
		\item[Case 2.]{If $\Delta=0$ then $\lambda_1=\lambda_2>0$ that is, the origin is a critical point \emph{repulsor}.
			
			For the points at infinity the system \eqref{sisinf} is reduced to:
			$$\begin{array}{l}
			\pm \dot{y}=(y-a-b)^2\\
			\pm \dot{z}=zy
			\end{array}$$
			
			with solution $y=a+b>0$. Then the phase portrait for the system in this case is:\\
			\begin{center}
				\includegraphics[width=5cm]{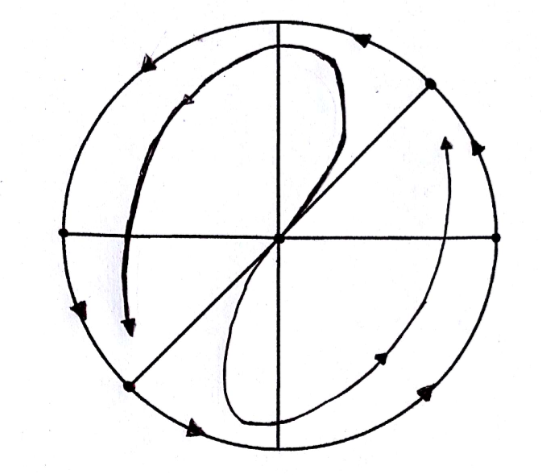} \\Global phase portrait for case 2 \eqref{sismin}
			\end{center}
		}
		\item[Case 3.]{If $\Delta<0$ then $\lambda_1=(a+b)+i\sqrt{c-2ab}$ and $\lambda_2=(a+b)-i\sqrt{c-2ab}$ with $(a+b)>0$ is to say that the origin is a critical point type \emph{Spiral}.  So the phase portrait for the system in this case is:\\
			
			\begin{center}
				\includegraphics[width=5cm]{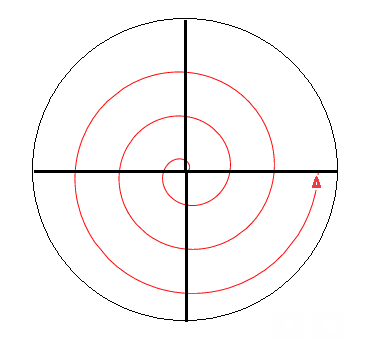} \\Global phase portrait for case 3 \eqref{sismin}
			\end{center}

			As regards infinity, the equation \eqref{evp2} has no real solutions, that is, it has no critical points at infinity.

		}
	\end{enumerate}
	
	
\end{proof}

\section{Bifurcations}

\begin{lemma}
	Let $B_0:=\{(a,b,c): a^2+b^2<-c \}$. if $(a,b,c)\in B_0$ the the associated system of the form \eqref{sismin} have unstable critical point $(0,0)$.   
\end{lemma}

\begin{proof}
		We consider the proper values \eqref{vpo}, if $(a,b,c)\in B_0$ it's to say $a^2+b^2<-c$, then $(a+b)^2<2ab-c$, it is $0<(a+b)<\sqrt{2ab-c}$.
	
	\begin{enumerate}
		\item[I.] If $a+b>0$, then $a+b<\sqrt{2ab-c}$, it's to say $$\lambda_2= \frac{a+b-\sqrt{2ab-c}}{2}<0 \hspace{1cm} \& \hspace{1cm} \lambda_1= \frac{a+b+\sqrt{2ab-c}}{2}>0.$$ This to implies that, the critical point $(0,0)$ is a saddle type.

		\item[II.] If $a+b<0$, then $-(a+b)<\sqrt{2ab-c}$, it's to say $$\lambda_1=\frac{a+b+\sqrt{2ab-c}}{2}>0 \hspace{1cm} \& \hspace{1cm} \lambda_2=\frac{a+b-\sqrt{2ab-c}}{2}<0.$$ Then the critical point $(0,0)$ is a saddle type. 
		
		\item[III.] If $a+b=0$ then $\lambda_1=\sqrt{2ab-c}>0$ and $\lambda_2=-\sqrt{2ab-c}<0$. Then the critical point $(0,0)$ is a sadle type.\\
	\end{enumerate}
\end{proof}

\begin{lemma}
	Let $B_1:=\{(a,b,c): 2ab-c<0 \hspace{0.2cm}\&\hspace{0.2cm} a+b>0 \}$. if $(a,b,c)\in B_1$ the the associated system of the form \eqref{sismin} have unstable focus critical point $(0,0)$.   
\end{lemma}

\begin{proof}
	We consider the proper values \eqref{vpo}, if $(a,b,c)\in B_1$, then $0<(a+b)^2>0>2ab-c$, it's to say that the proper values are: $$\lambda_1=\frac{a+b+i\sqrt{c-2ab}}{2}<0 \hspace{1cm} \& \hspace{1cm} \lambda_2=\frac{a+b-i\sqrt{c-2a}}{2}<0.$$
	where the real part in both cases is $(a+b)>0$, it's to say that the critical point is focus stable.
\end{proof}

\begin{lemma}
	Let $B_2:=\{(a,b,c): 2ab-c<0 \hspace{0.2cm}\&\hspace{0.2cm} a+b<0 \}$. if $(a,b,c)\in B_2$ the associated system of the form \eqref{sismin} have a critical point type saddle $(0,0)$.   
\end{lemma}
\begin{proof}
	The proof of this lemma is as a previous lemma, but in this case $a+b<0$, then the critical point is focus stable.
\end{proof}

\begin{lemma}
	Let $B_3:=\{(a,b,c): 2ab-c \geq 0 \hspace{0.2cm}\&\hspace{0.2cm} a+b>0 \hspace{0.2cm}\&\hspace{0.2cm} a^2+b^2>-c \}$. if $(a,b,c)\in B_3$ the the associated system of the form \eqref{sismin} have a critical point type unstable node $(0,0)$.   
\end{lemma}
\begin{proof}
	We consider the proper values \eqref{vpo}, if $(a,b,c)\in B_3$, then
	$a^2+b^2>-c$, it's to say $(a+b)^2>2ab-c>0$ ,
	it's to say $$\lambda_1=a+b+\sqrt{2ab-c}>0 \hspace{1cm} \& \hspace{1cm} \lambda_2=a+b-\sqrt{2ab-c}>0.$$ Then the critical point $(0,0)$ is a unstable node. 
	
\end{proof}

\begin{lemma}
	Let $B_4:=\{(a,b,c): 2ab-c\geq 0 \hspace{0.2cm}\&\hspace{0.2cm} a+b<0 \hspace{0.2cm}\&\hspace{0.2cm} a^2+b^2>-c\}$. if $(a,b,c)\in B_4$ the the associated system of the form \eqref{sismin} have a critical point type stable node $(0,0)$.   
\end{lemma}
\begin{proof}
	If $(a,b,c)\in B_4$, then $|a+b|\geq \sqrt{2ab-c}$, but $a+b<0$ it's to say $-(a+b)\geq \sqrt{2ab-c}$. Now the proper values are: $$\lambda_1=a+b+\sqrt{2ab-c}<0 \hspace{1cm} \& \hspace{1cm} \lambda_2=a+b-\sqrt{2ab-c}<0.$$ Then the critical point $(0,0)$ is a stable node. 
\end{proof}

\begin{prop}\label{BFRRA}
	The set $B:=\{(a,b,c): a+b<0 \hspace{0.2cm}\&\hspace{0.2cm} a^2+b^2=-c \} \cup  \{(a,b,c): a+b=0 \hspace{0.2cm}\&\hspace{0.2cm} a^2+b^2>-c \}$ is a bifurcation for the system \eqref{sismin}.
	
\end{prop}

\begin{proof}
For the previous lemmas, the set $B$ is a bifurcation for the family of systems \eqref{sismin}.	
\end{proof}

\begin{remark}
	We observe that this new bifurcation which appear in Proposition \eqref{BFRRA}, is interesting, beautiful and corresponds to one contribution of this paper.
\end{remark}

The following corollary summarizes the study of the bifurcations for the rest of families and is a direct consequence of the proposition \eqref{BFRRA}.

\begin{cor}
	The following statements hold:
	\begin{enumerate}
	\item For the family \eqref{f1}, $c=0$ is a bifurcation. For the family \eqref{f2} there is not a bifurcation. For the family \eqref{f3} $a=0$ is a bifurcation.
	
	\item For the families \eqref{f4} and \eqref{f5}, the set $F_1:=\{(b,c): b=0 \hspace{0.2cm}\&\hspace{0.2cm} b^2>-c \hspace{0.5cm} \text{or} \hspace{0.5cm} b^2=-c\hspace{0.2cm}\&\hspace{0.2cm} b>0\}$ is an bifurcation.
	
	\item  For the families \eqref{f6} and \eqref{f7}, the set $F_2:=\{(a,c): a=0 \hspace{0.2cm}\&\hspace{0.2cm} a^2>-c \hspace{0.5cm} \text{or} \hspace{0.5cm} a^2=-c\hspace{0.2cm}\&\hspace{0.2cm} a>0\}$ is a bifurcation.
	
	\item for the family \eqref{f8} the set $F_3:=\{(a,b): a+b=0\}$ is a bifurcation.
	\end{enumerate}
\end{cor}

\section{Galoisian Aspects}

For simplify the comprehension of the next proposition, we will summarize in a table all possibilities for parameter $\rho$.

\begin{table}[htbp]\label{tabla}
	\begin{center}
		\begin{tabular}{|c|c|c|c|c|}
			\hline
			 $a$ & $b$ & $c$ & $\rho$ & Differential Galois Group \\
			\hline \hline \hline \hline
			 0 & 0 & $\neq 0$ & $-c$ & $\mathbb{G}_m$\\ \hline
			 0 & $\neq 0$ & 0 & 0 &  $\mathbb{G}_a$\\ \hline
			 $\neq 0$ & 0 & 0 & 0 &  $\mathbb{G}_a$\\ \hline
			 $\neq 0$ & $\neq 0$ & 0 & $2ab$ &  $\mathbb{G}_m$\\ \hline
			 $\neq 0$ & 0 & $\neq 0$ & $-c$ &  $\mathbb{G}_m$\\ \hline
			 0 & $\neq 0$ & $\neq 0$ & $-c$ &  $\mathbb{G}_m$\\ \hline
			 $\neq 0$ & $\neq 0$ & $\neq 0$ & $2ab-c$ & $\mathbb{G}_m$\\ \hline 
		\end{tabular}
		\caption{Possibles $\rho$ values.}
		\label{tabla}
	\end{center}
\end{table}

\begin{prop}
For \eqref{sismin}, with $\rho=2ab-c$ we have that:\\
\begin{itemize}
	\item[i.] If $(2ab-c)=0$ then the differential Galois group is\\ $DGal<L/K>=(\{ \sigma_c : \sigma_c(y)=y+c \hspace{0.5cm} c\in\mathbf{C} \},o) \hspace{1cm} (1)$.
	\item[ii.]If $(2ab-c) \neq 0$ then the differential Galois group is\\ $DGal<L/K>=(\{ \sigma_c / \sigma_c(y)=cy \hspace{0.5cm} c\in\mathbf{C} \},o) \hspace{1cm} (2)$.
\end{itemize}

\end{prop}

\begin{proof}
	
	The associated foliation to \eqref{sismin} is: $$ yy'=(2a+2b)y-(a^2+c+b^2)x$$
	Then taking into account that $f(x)=-2(a+b)y$ and $g(x)=(a^2+c+b^2)x$,
	The Lienard equation will be: $$\ddot{x}-(2a+2b)\dot{x}+(a^2+c+b^2)x=0 $$
	This is a second order equation with constant coefficients.\\
	
	If we take it to the form $\ddot{y}=\rho y$ with the change of variable $x=\exp(-1/2\int f(x) dt)$ that is to say $x=\exp((a+b)t)y$, the equation is obtained: $$\ddot{y}=(2ab-c)y$$
	Let's analyze each case of the solution of the associated Lienard equation taking into account the following:\\
	$(2ab-c)=\rho$. The solution of the second order reduced equation is $y(t)=C_1\exp(\sqrt{\rho}t)+C_2\exp(-\sqrt{\rho}t)$ if they are taken $(C_1,C_2)=(1,0)$ or $(C_1,C_2)=(0,1)$ and the particular solutions are
	$$y_1=\exp(\sqrt{\rho}t) \hspace{1cm} y_2=\exp(-\sqrt{\rho}t).$$ With these we can build the solutions of the Lienard equation that are $$x_1=\exp((a+b)t) \hspace{1cm} x_2=\exp((a+b+\sqrt{2ab-c})t).$$ 
	
	For these we have the following cases:
	\begin{enumerate}
		\item[Case 1.]{If $a,b,c\in \mathbf{C}$ is this $(a+b+\sqrt{2ab-c})=\alpha+i\beta=z$, then $x_1=\exp(\alpha t).\exp(i\beta t)$
			\begin{enumerate}
				\item[1.1]{If $\alpha\geq0$ and $\beta=0$ then $x_1$ is not bounded in $-\infty$ and not bounded in $+\infty$}
				\item[1.2]{If $\alpha <0$ and $\beta=0$ then $x_1$ is not bounded in $+\infty$ and is not bounded in $-\infty$}
				\item[1.3]If {$\alpha=0$ and $\beta\neq0$ then the solution $x_1$ is periodic.}
			\end{enumerate}
		}
		\item[Case 2.]{If $a,b,c\in \mathbf{R}$ then $x_1=\exp((a+b)t).\exp(\sqrt{2ab-c}t)$}
		\begin{enumerate}
			\item[2.1]{If $c<2ab$ then we will have two solutions $x_1$ and $x_2$.}
			\item[2.2]{if $c>2ab$ then $(a+b)+\sqrt{2abc}=(a+b)+i\sqrt{c-2ab}$, that is to say that we would be in case one with  $\alpha=a+b$, $\beta=\sqrt{c-2ab}$ and also if $a=-b\Rightarrow\alpha=0$ or $a<-b\Rightarrow\alpha<0$ and $a>-b\Rightarrow\alpha>0$.}
		\end{enumerate}
	\end{enumerate}
	
	Our next aim will be to calculate the Galois groups associated with the equation $\ddot{y}=\rho y$, taking as a field of constants $K=\mathbf{C}$. This study can be seen in \cite{Ac3}. We must consider two cases:
	\begin{enumerate}
		\item[Case 1.]{ If $\rho=0$ then we will have the equation $\ddot{y}=0$, whose space of solutions is generated by $y_1=1$ and $y_2=t$. The Picard-Vessiot extension is $L=\mathbf{C}<t>$, where $<t>=\{\hat{k}_0+\hat{k}_1t: \hat{k}_0,\hat{k}_1\in\mathbf{C} \}$. If we take the automorphism $\sigma\in DGal<L/K>$ let's see what form you should have:\\
			First we must consider that $\sigma:L\rightarrow L$ so that $\sigma\mid_t =identity$, then $\sigma(y_1)=1=y_1$. Now
			
			\begin{displaymath}
			\sigma(\partial_t t)=\partial_t(\sigma(t))\\
			\sigma(1)=1=\partial_t(\sigma(t))
			\end{displaymath}
			
			Solving this equation by separable variables we have to $\sigma(t)=t+C$, therefore the differential Galois group will be $$DGal<L/K>=(\{ \sigma_c : \sigma_c(y)=y+c \hspace{0.5cm} c\in\mathbf{C} \},o)=\mathbb{G}_a.$$
			
			We can also see that:
			\begin{displaymath}
			\sigma
			\begin{pmatrix}
			y_1\\
			y_2
			\end{pmatrix}=
			\begin{pmatrix}
			1&0\\
			c_1&1
			\end{pmatrix}
			\end{displaymath}
			
			That is, to say that the differential Galois group is isomorphic to the group: 
			
			\begin{displaymath}
			\left\{
			\begin{pmatrix}
			1&0\\
			c_1&1
			\end{pmatrix}: c\in\mathbf{C} \}, \bullet \right\}
			\end{displaymath}
			
		}
		\item[Case 2.]{if $\rho\neq0$ we will then have the equation $\ddot(y)=\rho y$, taking the same field of constants from the previous case $K$. The base solutions for the equation are $y_1=\exp^{\sqrt{\rho}t}$ and $y_2=\exp(-\sqrt{\rho}t)=\frac{1}{y_1}$. the Picard-Vessiot extension will be $L=\mathbf{C}<\exp(\sqrt{\rho}t)>$.\\
			
			Now suppose that $\sigma\in DGal<L/K>$ and let's see what action he performs on the solutions $y_1,y_2$.
			\begin{displaymath}
			\sigma(y_1)=\sigma(\exp(\sqrt{\rho}t))\Leftrightarrow\\
			\partial_t\sigma(y_1)=\sigma(\partial_t\exp(\sqrt{\rho}t))=\sqrt{\rho}(y_1)\\
			\end{displaymath}
			
			Solving the differential equation:$\partial_t\sigma(y_1)=\sqrt{\rho}(\sigma(y_1))$, by the method of separation of variables we have left $\sigma(y_1)=y_1.c$. also as $y_2=1/y_1$, then $\sigma(y_2)=\frac{1}{\sigma(y_1)}=\frac{1}{c.y1}=\frac{1}{c}y_2$.\\
			
			The Galois differential group will be $$DGal<L/K>=(\{ \sigma_c / \sigma_c(y)=cy \hspace{0.5cm} c\in\mathbf{C} \},o)=\mathbb{G}_m.$$
			
			We can also see in this case that:
			\begin{displaymath}
			\sigma
			\left(\begin{array}{l}
			y_1\\
			y_2
			\end{array} \right)=
			\begin{pmatrix}
			c&0\\
			0&1/c
			\end{pmatrix}
			\end{displaymath}
			
			That is, to say that the Galois differential group is isomorphic to the group:
			\begin{displaymath}
			\left\{
			\begin{pmatrix}
			c&0\\
			0&1/c
			\end{pmatrix}
			: c\in\mathbf{C}, \bullet \right\}
			\end{displaymath}
			
		}

	\end{enumerate}
\end{proof}

\section{Darboux theory of integrability}

\begin{prop}
	Consider the values of $\rho$ according to the table \ref{tabla}. The following statements hold:
\begin{itemize}
	\item[i.]{The invariant algebraic curves are $f_1=-v+\sqrt{\rho}$ y $f_2=-v-\sqrt{\rho}$ and their respective generalized cofactors $K_1=-1(v+\sqrt{\rho})$ y $K_2=-1(v-\sqrt{\rho})$.}
	\item[ii.]{The Generalized Exponential factors are $F_1(v,x)=\exp({\sqrt{\rho}x+c})$\\ and $F_2(v,x)=\exp({-\sqrt{\rho}x+c})$. Moreover, their respective generalized cofactors are
		\\ $L_1=(v,x)=\sqrt{\rho}$ y $L_2=-\sqrt{\rho}$ }
	\item[iii.]{The integrating factors are $R_1(v,x)=\frac{\exp({-2\sqrt{\rho}x})}{(-v+\sqrt{\rho})^2}$ and $R_2(v,x)=\frac{\exp({2\sqrt{\rho}x})}{(-v-\sqrt{\rho})^2}$}
	\item[iv.]{The first integrals are given by
	\begin{displaymath}
	I_1(v,x)
	=\frac{-v-\sqrt{\rho}-\frac{\exp({-2\sqrt{\rho}x})}{-2\sqrt{\rho}\exp({-2\sqrt{\rho}x}+C)}}{-v-\sqrt{\rho}}(-2\sqrt{\rho}\exp({-2\sqrt{\rho}x}))\exp(1/4\rho)
		\end{displaymath}
		
	\begin{displaymath}
	I_2(v,x)
	=\frac{-v+\sqrt{\rho}-\frac{\exp(2\sqrt{\rho}x)}{2\sqrt{\rho}\exp(2\sqrt{\rho}x)+C}}{-v+\sqrt{\rho}}(2\sqrt{\rho}\exp(2\sqrt{\rho}x))\exp(1/4\rho)
		\end{displaymath}  }
	
\end{itemize}

\end{prop}

\begin{proof}
	If we change the variable $v=\frac{\dot{y}}{y}$ about the equation $\ddot{y}=\rho y$ we will have the Riccati equation:
$$\dot{v}=\rho-v^2$$
With solutions corresponding to $y_1, y_2$ respectively $$v_1=\sqrt{\rho}, \hspace{1cm} v_2=-\sqrt{\rho}$$

Taking this equation as a foliation, the associated system will be:
\begin{equation}
\begin{array}{ccl}
\dot{x}&=&1 \bigskip \\
\dot{v}&=&\rho-v^2
\end{array}
\end{equation}

and the associated vector field is $X=\partial_x+(\rho-v^2)\partial_v$.

Applying Lemma 1 of \cite{Ap}, we identify each of the Darboux integrability elements (also defined in Section 1.2), with $v_\lambda (x)$ as the solution: \\

\begin{enumerate}
	
	\item[a.]{\textbf{invariant algebraic curves:}\\ ($f_\lambda(v,x)=-v+v_\lambda(x)$)\\
		$f_1=-v+\sqrt{\rho}$ \hspace{1cm} and \hspace{1cm} $f_2=-v-\sqrt{\rho}$
		
		\textbf{Generalized Cofactors:}\\ ($K_\lambda(v,x)=-N(x)(v+v_\lambda(x))$) in this case:\\
		$K_1=-1(v+\sqrt{\rho})$ \hspace{1cm} and \hspace{1cm} $k_2=-1(v-\sqrt{\rho})$
	}
	
	\item[b.]{\textbf{Generalized Exponential factor:}\\ ($F_\lambda(v,x)=\exp^{\int(\frac{N'(x)}{2N(n)}+v_\lambda(x))dx}$) in this case: \bigskip\\
		$F_1(v,x)=\exp(\int(\frac{0}{2}+v_1(x))dx)=\exp(\int \sqrt{\rho}dx)=\exp(\sqrt{\rho}x+c)$\\
		also\\
		$F_2(v,x)=\exp(\int(\frac{0}{2}+v_2(x))dx)=\exp(\int -\sqrt{\rho}dx)=\exp(-\sqrt{\rho}x+c)$\\
		
		\textbf{Generalized Cofactor:}\\ $L_\lambda(v,x)=N'(x)/2+N(x)v_\lambda(x)$, in this case:\bigskip\\
		$L_1=(v,x)=N'/2+Nv_1=\sqrt{\rho}$,\hspace{1cm} $L_2=(v,x)=N'/2+Nv_2=-\sqrt{\rho}$}
	
	\item[c.]{\textbf{Integrating Factor:}\\ 		$R_\lambda(v,x)=\frac{\exp(\int(-N'/N-2v_\lambda)dx)}{-v+v_\lambda}$ in this case:\\
		$$R_1(v,x)=\frac{\exp(\int(-1'/1-2v_1)dx)}{-v+v_1}=\frac{\exp(-2\sqrt{\rho}x)}{(-v+\sqrt{\rho})^2}$$
		
		$$R_2(v,x)=\frac{\exp(\int(-1'/1-2v_2)dx)}{-v+v_2}=\frac{\exp(2\sqrt{\rho}x)}{(-v-\sqrt{\rho})^2}$$
	}
	\item[d.]{\textbf{First Integral}
		For this case we must first calculate two elements:
		\begin{enumerate}
			\item[1.]{$v_{(\lambda,1)}=(\ln(y_{\lambda}))'$, In our case\\  $v_{(1,1)}=\sqrt{\rho}$ \hspace{1cm} and \hspace{1cm} $v_{(2,1)}=\sqrt{\rho}$.}
			\item[2.]{$v_{(\lambda,2)}=v_{(\lambda,1)}+\frac{\exp(-2\int v_{(\lambda,1)}dx)}{\int (\exp(-2\int v_{(\lambda,1)}dx))dx}$, In our case
				
				\begin{displaymath}
				\begin{array}{l}
				v_{(1,2)}=v_{(1,1)}+\dfrac{\exp^{-2\int v_{(1,1)}dx}}{\int (\exp(-2\int v_{(1,1)}dx))dx}\bigskip\\
				
				=\sqrt{\rho}+\dfrac{\exp(-2\int \sqrt{\rho}dx)}{\int (\exp(-2\int \sqrt{\rho}dx))dx}\bigskip\\
				=\sqrt{\rho}-\dfrac{\exp(-2\sqrt{\rho}x)}{-2\sqrt{\rho}\exp(-2\sqrt{\rho}x)+C}
				\end{array}
				\end{displaymath}
				
				\begin{displaymath}
				\begin{array}{l}
				v_{(2,2)}=v_{(2,1)}+\dfrac{\exp(-2\int v_{(2,1)}dx)}{\int (\exp(-2\int v_{(2,1)}dx))dx}\bigskip\\
				=-\sqrt{\rho}+\dfrac{\exp(2\int \sqrt{\rho}dx)}{\int (\exp(2\int \sqrt{\rho}dx))dx}\bigskip\\
				=-\sqrt{\rho}+\frac{\exp(2\sqrt{\rho}x)}{2\sqrt{\rho}\exp(2\sqrt{\rho}x)+C}
				\end{array}
				\end{displaymath}
				
			}
		\end{enumerate}
		
		Now the first integrals are given by:
		$$I_\lambda(v,x)=\dfrac{-v-v_{(\lambda,2)}}{-v-v_{(\lambda,1)}}\exp(\int(v_{(\lambda,2)}-v_{(\lambda,1)})dx)$$
		Then:
		\begin{displaymath}
		\begin{array}{l}
		I_1(v,x)=\bigskip\\
		\dfrac{-v-\sqrt{\rho}-\dfrac{\exp(-2\sqrt{\rho}x)}{-2\sqrt{\rho}\exp(-2\sqrt{\rho}x)+C}}{-v-\sqrt{\rho}}\exp(\int(\sqrt{\rho}-\dfrac{\exp(-2\sqrt{\rho}x)}{-2\sqrt{\rho}\exp(-2\sqrt{\rho}x)+C}-\sqrt{\rho})dx)\bigskip\\
		=\dfrac{-v-\sqrt{\rho}-\dfrac{\exp(-2\sqrt{\rho}x)}{-2\sqrt{\rho}\exp(-2\sqrt{\rho}x)+C}}{-v-\sqrt{\rho}}(-2\sqrt{\rho}\exp(-2\sqrt{\rho}x))\exp(1/4\rho)
		\end{array}
		\end{displaymath}
		
		\begin{displaymath}
		\begin{array}{l}
		I_2(v,x)=\bigskip\\
		\dfrac{-v+\sqrt{\rho}-\dfrac{\exp(2\sqrt{\rho}x)}{2\sqrt{\rho}\exp(2\sqrt{\rho}x)+C}}{-v+\sqrt{\rho}}\exp(\int(-\sqrt{\rho}+\dfrac{\exp(2\sqrt{\rho}x)}{2\sqrt{\rho}\exp(2\sqrt{\rho}x)+C}+\sqrt{\rho})dx)\bigskip\\
		=\dfrac{-v+\sqrt{\rho}-\dfrac{\exp(2\sqrt{\rho}x)}{2\sqrt{\rho}\exp(2\sqrt{\rho}x)+C}}{-v+\sqrt{\rho}}(2\sqrt{\rho}\exp(2\sqrt{\rho}x))\exp(1/4\rho)
		\end{array}
		\end{displaymath}
		
	}
\end{enumerate}

\end{proof}


\section{Final Remarks}
In this paper we studied from algebraic and qualitative point of view one parametric family of linear differential systems. Such parametric family comes from the correction of Exercise 11 in \cite[\S 1.3.3]{PY}. Which we called  Polyanin-Zaitsev vector field, see \cite{ARR}.\\

In this case we taken a particular member of a family of equations, which were studied on general case in \cite{ARR}. We had found critical points and the description near to these points. We obtained a new type of bifurcation that it seems new in the literature. In the algebraic approach, the explicit first integral has been found using the Darboux method. Moreover the differential Galois groups associated to solutions also were found.


\begin{thebibliography}{18}
%

\bibitem{ARR} P. B. Acosta-Hum\'anez, A. Reyes-Linero and J. Rodr\'iguez-Contreras, {\it Algebraic and qualitative remarks about the family $yy'= (\alpha x^{m+k-1} + \beta x^{m-k-1})y + \gamma x^{2m-2k-1}$}, preprint 2014.  Available at arXiv:1807.03551. 

\bibitem{PY} A.D. Polyanin and V.F. Zaitsev, {\it Handbook of exact solutions for ordinary differential equations, Second Edition}. Chapman and Hall, Boca Raton (2003).

\bibitem{Ap} P.B. Acosta-Hum\'anez, Ch. Pantazi, {\it Darboux Integrals for Schrodinger Planar Vector Fields via Darboux Transformations} SIGMA, 8 (2012) 043.

\bibitem{VDP} Van der Pol, B., and Van der Mark, J., {\it Frequency demultiplication}, Nature, 1927, 120, 363-364.

\bibitem{119} Nagumo, J., Arimoto, S. and Yoshizawa, S. {\it An active pulse transmission line simulating nerve axon}, Proc. IRE, 1962, 50, 2061-2070.

\bibitem{120} Guckenheimer, J., Hoffman, K., and Weckesser, W., {\it The forced Van der Pol equation I: The slow flow and its bifurcations}, SIAM J. Applied Dynamical Systems, 2003, 2, 1-35.

\bibitem{121} Kapitaniak, T., {\it Chaos for Engineers: Theory, Applications and Control}, Springer, Berlin, Germany, 1998.

\bibitem{Almp}P.B. Acosta-Hum\'anez, J.T L\'azaro, J.J. Morales-Ruiz, Ch. Pantazi, {\it On the integrability of polynomial fields in the plane by means of Picard-Vessiot theory}. Preprint arXiv:1012.4796.


\bibitem{PK} L. Perko, {\it Differential equations and Dynamical systems, Third Edition}. Springer-Verlag  New York, Inc (2001).

\bibitem{Jg} J. Guckenheimer, {\it Nonlinear Oscillations, Dynamical Systems and Bifurcations of Vector Fields}. Springer-Verlag  New York (1983).

\bibitem{Ns} V.V. Nemytskii and V.V. Stepanov, {\it Qualitative Theory of Differential Equations}, Princeton University Press, Princeton (1960).

\bibitem{Cll} C. Christopher \& J. Llibre, {\it Integrability via invariant algebraic curves for
	planar polynomial differential systems, Annals of Differential Equations} 14
(2000), 5-19.

\bibitem{CPP} Ch. Pantazi, {\it Inverse problems of the Darboux Theory of integrability for
	planar polynomial differential systems}, PhD, 2004.


\bibitem{A} P.B. Acosta-Hum\'anez, {\it Teor\'{\i}a de Morales-Ramis y el algoritmo de Kovacic} Lecturas Matem\'aticas
Volumen Especial (2006) 2156.

\bibitem{Ac3} P.B. Acosta-Hum\'anez, {\it Galoisian Approach to Supersymmetric Quantum Mechanics}. PHD. Thesis barcelona (2009).  Available at arXiv:0906.3532

\bibitem{Amw} P. Acosta-Hum\'anez, J. Morales-Ruiz and J.-A. Weil, {\it Galoisian Approach to integrability of Schrödinger Equation}. Reports on Mathematical Physics, 67 (2011) 305-374.

\bibitem{AJp} P.B. Acosta-Hum\'anez, J. Perez, {\it Teor\'{\i}a de Galois diferencial: una aproximaci\'on} Lecturas Matem\'aticas
Vol\'umen Especial (2006) 2156.


\bibitem{AJp2} P.B. Acosta-Hum\'anez, J. Perez, {\it Una introducci\'on teor\'ia de Galois diferencial} Bolet\'in de Matem\'aticas Nueva Serie, 11 (2004) 138149.

\bibitem{VDPS} M. van der Put and M. Singer, {\it Galois Theory in Linear Differential Equations} . Springer-Verlag  New York (2003).

\bibitem{JJM} J. Morales-Ruiz, {\it Differential Galois Theory and Non-Integrability of Hamiltonian Systems}. Birkhäuser, Basel (1999).

\bibitem{Lang} S. Lang, {\it Linear Algebra}, Undergraduate Text in Mathematics, Springer, Third edition (2010).


\bibitem{JAW} J.A. Weil, {\it Constant et polyn\'omes de Darboux en algèbre diff\'erentielle: applications aux systèmes diff\'erentiels lin\'eaires}. Doctoral thesis, (1995).

\bibitem{P2} Ch. Pantazi, {\it El metodo de Darboux, en notas del primer seminario de Integrabiliad}. Universidad Politecnica de Catalunya  (2005). Available at $https://upcommons.upc.edu/e-prints/bitstream/2117/2233/1/noinupc.pdf$.

\bibitem{GGG} H.Giacomini, J. Gine and M. Grau, {\it Integrability of planar polynomial differential systems through linear differential equations}. Birkhauser, Rocky Mountain J. Math. 36 (2006) 457485.










\end{thebibliography}
\end{document}